\documentclass[12pt]{amsart}
\usepackage[colorlinks=true,citecolor=blue,linkcolor=magenta]{hyperref}
\usepackage{amsmath}
\usepackage{verbatim}
\usepackage{amsfonts}
\usepackage{amssymb}
\usepackage{blkarray}
\usepackage{color,colortbl}
\usepackage[all]{xy}  
\usepackage{enumerate}
\usepackage[top=1in, bottom=1in, left=1in, right=1in]{geometry}
\usepackage{mathrsfs} 
\usepackage{stmaryrd}
\usepackage[nameinlink]{cleveref}
\usepackage{soul}
\usepackage[colorinlistoftodos]{todonotes}
\usepackage{colonequals}
\usepackage{comment}

\theoremstyle{theorem}
\newtheorem{theorem}{Theorem}[section]

\numberwithin{equation}{section}

\newtheorem{question}[theorem]{Question}

\newtheorem{lemma}[theorem]{Lemma}
\newtheorem{proposition}[theorem]{Proposition}

\newtheorem{problem}[theorem]{Problem}

\newtheorem{conjecture}[theorem]{Conjecture}

\theoremstyle{definition}
\newtheorem{definition}[theorem]{Definition}

\newtheorem{example}[theorem]{Example}

\newtheorem{remark}[theorem]{Remark}

\newcommand{\reg}{\operatorname{reg}}

\newcommand{\Ass}{\operatorname{Ass}}

\newcommand{\Char}{\operatorname{char}}
\newcommand{\Min}{\operatorname{Min}}

\newcommand{\het}{\operatorname{ht}}
\newcommand{\bight}{\operatorname{bight}}

\newcommand{\ch}{\operatorname{char}}

\newcommand{\N}{\mathbb{N}}
\renewcommand{\P}{\mathbb{P}}
\newcommand{\Q}{\mathbb{Q}}

\newcommand{\m}{\mathfrak{m}}

\DeclareMathOperator{\R}{\mathcal{R}}
\newcommand{\sra}{\mathcal{R}_s}
\DeclareMathOperator{\gt}{gt}
\DeclareMathOperator{\svd}{svd}
\DeclareMathOperator{\lcm}{lcm}

\title{Symbolic Rees algebras}

\author[Elo\'isa Grifo]{Elo\'isa Grifo}
\address{Department of Mathematics, University of California, Riverside, Riverside, CA 92521, USA}
\email{eloisa.grifo@ucr.edu}

\author[Alexandra Seceleanu]{Alexandra Seceleanu}
\address{Department of Mathematics, University of Nebraska -- Lincoln, Lincoln, NE 68588, USA}
\email{aseceleanu@unl.edu}

\dedicatory{Dedicated to David~Eisenbud on the occasion of his 75th birthday.}

\subjclass[2010]{Primary: 13A15. Secondary: 13H05}
\keywords{symbolic powers, symbolic Rees algebra, containment problem}

\begin{document}

\begin{abstract}
We survey old and new approaches to the study of symbolic powers of ideals. Our focus  is on the symbolic Rees algebra of an ideal, viewed both as a tool to investigate its symbolic powers and as a source of challenging problems in its own right. We provide an invitation to this area of investigation by stating several open questions.
\end{abstract}

\maketitle

\setcounter{tocdepth}{1} 
\tableofcontents

\section{Introduction}\label{section intro}

Symbolic powers arise from the theory of primary decomposition. It is often surprising to the novice algebraist that the powers of an ideal can acquire associated primes that were not associated to the ideal itself. In that sense, the symbolic powers of $I$ are more natural.

\begin{definition}
\label{def:symbolicpower}
Let $R$ be a noetherian ring and $I$ an ideal in $R$ with no embedded primes. 
The {\bf $n$-th symbolic power} of $I$ is the ideal 
$$I^{(n)} \quad \colonequals \bigcap_{P \in \Ass(R/I)} I^n R_p \cap R.$$
\end{definition}
This is the ideal obtained by intersecting the components in a primary decomposition of $I$ corresponding to the associated primes of $I$, which by assumption are all minimal. When $I$ does have embedded primes, there are two possible definitions of symbolic power to chose from: either taking $P$ to range over the associated primes of $I$, or over the minimal primes of $I$. To avoid this, we will focus on the case of ideals with no embedded primes. Note that the symbolic powers of a prime ideal are already very interesting, and thus our assumption that $I$ has no embedded primes is fairly mild. 

When $I$ is a radical ideal in $R = k[x_1, \ldots, x_d]$, where $k$ is a perfect field, $I^{(n)}$ coincides with the set of polynomials that vanish to order $n$ on the variety defined by $I$ \cite{Zariski,Nagata,EisenbudHochster}. In general, we always have $I^{(1)} = I$, by definition, and it is easy to show that $I^n \subseteq I^{(n)}$ always holds. However, given an ideal $I$ and some $n > 1$, determining whether the equality $I^n = I^{(n)}$ holds can be a very difficult question. This stems from the fact that computing primary decompositions is a difficult problem; as Decker, Greuel, and Pfister write in \cite{AlgorithmsPrimDec}, ``providing efficient algorithms for primary decomposition of an ideal [...] is [...] still one of the big challenges for computational algebra and computational algebraic geometry". In fact, even if one restricts to monomial ideals, the problem of finding a primary decomposition is NP-complete \cite{MonomialNP}. This is one of the reasons why many innocent sounding questions one could ask about symbolic powers remain open. 

Nevertheless, there exist sufficiently efficient methods for computation of symbolic power ideals using computer algebra systems such as {\em Macaulay2} \cite{M2}. Some of these methods are used in the Macaulay2 package {\em SymbolicPowers}; we refer to \cite{SymbolicPowersPackage} for an account of the functionality offered by this package.

Symbolic powers are ubiquitous throughout commutative algebra, with connections to virtually all topics in the field. For a more general survey on symbolic powers, we direct the reader to \cite{SurveySP}. In this survey, we focus on symbolic Rees algebras.

\section{Symbolic Rees algebras}

The symbolic powers of $I$ form a graded family of ideals, meaning that $I^{(a)} I^{(b)} \subseteq I^{(a+b)}$ for all $a$ and $b$. Thanks to this simple property, we can package together all the symbolic powers of $I$ to form a graded ring. This is the so called symbolic Rees algebra of $I$, which contains much information about $I$ and its symbolic powers, and the main character in this survey.

\begin{definition}[Symbolic Rees algebra]
	Let $R$ be a noetherian ring and $I$ an ideal in $R$. The {\bf symbolic Rees algebra} of $I$, also known as the {\bf symbolic blow-up ring} of $I$, is the graded ring
	$$\sra(I) \colonequals R[It, I^{(2)}t^{(2)}, \ldots]=\bigoplus_{n \geqslant 0} I^{(n)} t^n \subseteq R[t].$$
	The indeterminate $t$ of degree one is helpful in keeping track of the degree of elements in the symbolic Rees algebra. It helps distinguish an element $f\in I^{(n)}$, which we write $ft^n$, from the element $f\in I$, which we write $ft$.
\end{definition}

This construction is akin to that of the Rees algebra of $I$, which is the graded ring 
$$\mathcal{R}(I) \colonequals R[It, I^{2}t^2, \ldots]=\bigoplus_{n \geqslant 0} I^n t^n \subseteq R[t].$$

The study of Rees algebras is very rich and presents its own challenges (see \cite{VasconcelosBook95} for an overview), and yet the symbolic Rees algebra of $I$ is often much more complicated than the ordinary Rees algebra. While the study of symbolic Rees algebras is certainly inspired by Rees algebras, there is a crucial difference: $\mathcal{R}(I)$ is a finitely generated $R$-algebra, while $\sra(I)$ may fail to be an algebra-finite extension of $R$. Indeed, the Rees algebra of $I$ is generated over $R$ in degree $1$, by a (finite) generating set of $I$, that is, $\R(I)$ is a standard graded noetherian ring, i.e., generated as an $R$-algebra by elements of degree 1. In contrast, the symbolic Rees algebra may require infinitely many generators. As we will see, the symbolic Rees algebra of $I$ is a finitely generated $R$-algebra if and only if $\R_s(I)$ is a noetherian ring. Even if $\sra(I)$ is noetherian, it may be generated in different degrees; we introduce the generation type in \Cref{s:gtsvd} to quantify this. The symbolic Rees algebra of $I$ is generated in degree $1$ precisely if $I^n = I^{(n)}$ for all $n\in\N$, in which case $\mathcal{R}(I)$ and $\sra(I)$ coincide. Sufficient criteria for this equality are presented in \cite{HochsterCriteria} and \cite{RobVal}.
\subsection{A brief history} 

Although symbolic Rees algebras appear implicitly in the 1950s in work of Rees, Zariski, Nagata, and others surveyed below, this class of algebras did not acquire a name until several decades later. To our knowledge, the terminology ``symbolic Rees algebra"  appears for the first time in Huneke's paper \cite{HunekeCriterion} in 1982, while the monograph \cite{VasconcelosBook95} by Vasconcelos proposes the alternative terminology ``symbolic blowup algebra".

The first example of an ideal whose symbolic Rees algebra is not finitely generated appears in Rees' counterexample to Zariski's Formulation of Hilbert's 14th Problem (\Cref{q:Hilbert14}). 

\begin{question}[Hilbert's 14th Problem]
\label{q:Hilbert14}
	Let $k$ be a field. For all $n \geqslant 1$, and all subfields $K$ of $k(x_1, \ldots, x_n)$, is $K \cap k[x_1, \ldots, x_n]$
	finitely generated over $k$?
\end{question}

An important special case that provided the original motivation for this question concerns the ring of invariants of a linear action of a group of matrices on a polynomial ring over a field. For $R=k[x_1,\ldots, x_n]$, a polynomial ring with coefficients in a field $k$ equipped with a linear  action of a group $G\subseteq {\rm GL}_n(k)$, one studies the subring of $G$-invariant polynomials
$$R^G=\{f\in R \mid g \cdot f = f \text{ for all } g \in G\}.$$
 A fundamental question in invariant theory is whether $R^G$ is finitely generated as a $k$-algebra. For finite groups, an affirmative answer is due to E.~Noether \cite{Noether}. The finite generation of $R^G$ is the particular case of \Cref{q:Hilbert14} where $K$ is the subfield of elements of the fraction field of $R$ fixed by $G$.

 We point the reader to the surveys \cite{Mumford, Survey14} for more on Hilbert's 14th problem, and we will instead focus on the connections between symbolic Rees algebras and this famous question. The foundation of this connection was laid by Zariski in the early 1950s in \cite{Zariski54} by interpreting the rings $K \cap k[x_1, \ldots, x_n]$ as rings of rational functions on a nonsingular projective variety $X$ with poles restricted to a specified divisor $D$. Such varieties $X$ can be obtained geometrically by the procedure of blowing up, and $D$ is usually taken to be the exceptional divisor of the blow up $X$.

Zariski \cite{Rees1958ProblemZariski} formulated a more general version of \Cref{q:Hilbert14}, by taking any integrally closed domain that is finitely generated over $k$ in place of $R=k[x_1, \ldots, x_d]$. The first counterexample to Zariski's version of \Cref{q:Hilbert14} was given by Rees \cite{Rees1958ProblemZariski}, and this is where the connection with symbolic Rees algebras first appears. The crux of Rees' proof, while not written in the language of symbolic Rees algebras, consists of showing that if $P$ is a height $1$ prime ideal in the affine cone over an elliptic curve with infinite order in the divisor class group, then its symbolic Rees algebra is not finitely generated. We give a numerical example to illustrate the principles used by Rees.
 \begin{example}[Rees]
 Consider the elliptic curve $C$ cut out by the equation $x^3-y^2z-2z^3$ in the projective plane $\P_{\Q}^2$. The point $p=(3,5,1)$ is a rational point on this curve which has infinite order with respect to the group law on $C$ \cite[Example 2.4.6(3)]{EllipticCurvesBook}. Consider the coordinate ring $R=\Q[x,y,z]/(x^3-y^2z-2z^3)$ of $C$ and the ideal $P=(x-3z, y-5z)$ of $R$ which defines $p$. Then \cite{Rees1958ProblemZariski} yields that $\R_S(P)$ is not a finitely generated $\Q$-algebra. 
 
 By contrast, consider the point $q=(2,3,1)$ on the elliptic curve with coordinate ring $S= \Q[x,y,z]/(x^3-y^2z+z^3)$. The point $q$ has order six with respect to the group law of this curve \cite[Example 2.4.2]{EllipticCurvesBook}, and examining the ideal $Q=(x-2z, y-3z)$ defining this point with Macaulay2 \cite{M2} yields
\[
Q^{(6)}=(12x^2-6xy+y^2-6xz-6yz+9z^2),
\]
which is a principal ideal. Moreover, $Q^{(6n)}=(Q^{(6)})^n$ for all $n\geqslant 0$, which as we will see in \Cref{noetherian equivalences} implies that $\R_S(Q)$ is  a finitely generated $\Q$-algebra. 
 \end{example}

 In the late 1950s, Nagata found the first example of an ideal $I$ in a polynomial ring whose symbolic Rees algebra is not finitely generated, giving a counterexample to Hilbert's 14th Problem \cite{Nagata59}. In fact, he constructed a ring of invariants which is not a finitely generated algebra. The ideal constructed by Nagata defines a set of 16 points in the projective plane, and hence is not a prime ideal like in the example provided by Rees. Nagata's method is to relate the structure of $\R_s(I)$ to an interpolation problem in the projective plane, namely, that for each $m \geqslant 1$, there does not exist a curve of degree $4m$ having multiplicity at least $m$ at each of 16 general points of the projective plane.

In the 1980s, Roberts constructed new examples of symbolic Rees algebras that are not finitely generated based on Nagata's examples. His work shows that $\sra(I)$ may fail to be finitely generated even when $I$ is a prime ideal in a regular ring \cite{RobertsExample}, thus answering a question of Cowsik in the negative \cite{Cowsik}. Roberts' examples are prime ideals in a polynomial ring over a field of characteristic $0$, and later Kurano \cite{Kurano} showed that if we consider the same examples in prime characteristic $p$, their symbolic Rees algebras are in fact Noetherian. Roberts' examples \cite{RobertsExample}, while prime, are not analytically irreducible, meaning that these prime ideals do not stay prime after passing to the completion; he later improved this by providing an example that was in fact analytically irreducible \cite{RobertsExample2}. 

Still, as shown below, finite generation has powerful consequences for some symbolic Rees algebras, and thus it is natural to ask when this occurs. Huneke gave a general criterion for a symbolic Rees algebra of a height $2$ prime ideal in a $3$-dimensional regular local ring to be finitely generated \cite{HunekeCriterion,HunekeHilbertSymb}, which we will discuss in more detail in \Cref{subsection criteria}.

In the early 1980s, Cowsik showed that if $I$ determines a curve in $\mathbb{A}^n_k$, where $k$ is an infinite field, and the symbolic Rees algebra of $I$ is finitely generated, then $I$ is a set-theoretic complete intersection \cite{Cowsik}. This has been exploited to show that certain curves are indeed set-theoretic complete intersections in \cite{EliahouCurve}. A modern generalization of Cowsik's result states that, if the symbolic Rees algebra of an ideal $I$ is finitely generated, the arithmetic rank of $I$, that is, the least number of generators of an ideal whose radical agrees with the radical of $I$, is bounded above by the polynomial order of growth for the number of generators of $I^{(n)}$ as a function of $n$; see \cite[Proposition 2.3]{DaoMontano}.

One interesting case is that of space monomial curves, which are Zariski closures of images of maps of the form $\mathbb{A}^1\to \mathbb{A}^3, t\mapsto (t^a, t^b, t^c)$. We abbreviate this by referring to a monomial curve as $(t^a, t^b, t^c)$. The defining ideals of space monomial curves were known to be set theoretic complete intersections since 1970 \cite{Herzog1970}, and thus one could hope that in fact their symbolic Rees algebras are always finitely generated. This is, however, false: Goto, Nishida, and Watanabe \cite{NonnoetherianSymb} found the first counterexamples, a family of choices of $(a, b, c)$ which give infinitely generated symbolic Rees algebras over a field in characteristic $0$. We record this interesting family of examples below.

\begin{example}[Goto--Nishida--Watanabe]
\label{ex:monomialcurve}
Let $P$ be the defining ideal in the power series ring $k \llbracket x,y,z \rrbracket$ over a field $k$ of the space monomial curve 
$$x=t^{7n-3}, \, y=t^{(5n-2)n}, \, z=t^{8n-3} \quad  \text{ where } n \geqslant 4 \textrm{ and } n\not \equiv 0 \!\!\!\!\pmod{3}.$$
For example, when $n=4$, our curve is parametrized by $x = t^{25}$, $y = t^{72}$, $z = t^{29}$, and
$$P=(y^3-x^4z^4,x^{11}-yz^7,x^7y^2-z^{11}).$$
Then $\R_s(P)$ is a non-Cohen-Macaulay noetherian ring if $\ch k>0$, and $\R_s(P)$ is not a noetherian ring if $\ch k=0$.
\end{example}

Other examples where $\sra(P)$ is not Cohen-Macaulay with $P$ the defining ideal of $(t^a, t^b, t^c)$ in prime characteristic were already known by \cite{MorimotoGoto}. 

The late 1980s and early 1990s saw a program to classify when these symbolic Rees algebras are (or are not) finitely generated \cite{Morales,SchenzelExamples,HemaMonomialCurves}. Most notably, Cutkosky gave criteria which say, for example, that over any field, $\sra(P)$ is finitely generated whenever $(a+b+c)^2 > abc$. Cutkosky's work \cite[Lemma 7]{Cutkosky} also uncovered a deep connection to a different geometric problem: that $\sra(P)$, where $P$ defines $(t^a,t^b,t^c)$, is noetherian if and only if a certain space --- the blow-up at a general point of the weighted projective space $\P(a,b,c)$ --- is a Mori dream space --- meaning its Cox ring is noetherian. Using this connection, Gonz\'alez Anaya, Gonz\'alez, and Karu \cite{GonzalezKaru16,GonzalezKaru2,GonzalezGonzalezKaru,GonzalezGonzalezKaru2,GonzalezGonzalezKaru3} have more recently found several large families of examples in characteristic $0$ that in particular recover the original family of examples of Goto--Nishida--Watanabe of non-noetherian $\sra(P)$; in fact, they give a complete characterization of when $\sra(P)$ is (non)noetherian for large families of curves of type $(t^a,t^b,t^c)$. The smallest of their examples to date are the curves $(t^7,t^{15},t^{26})$ and $(t^{12},t^{13},t^{17})$, each of these examples being smallest in a different manner. For more examples of this kind, see also \cite{HeMori}. The original family of non-noetherian examples in \cite{NonnoetherianSymb} has also been generalized via different methods in \cite{JohnsonReedMonomialCurves}.

Finally, the story of symbolic Rees algebras of space monomial curves has deeper connections to Hilbert's 14th Problem: Kurano and Matsuoka showed that whenever the symbolic Rees algebra of the defining ideal $P$ of $(t^a,t^b,t^c)$ is not noetherian, then in fact $\sra(P)$ is a counterexample to Hilbert's 14th Problem \cite{KuranoMatsuoka}.

Even when the symbolic Rees algebra $\sra(P)$ of the curve $(t^a,t^b,t^c)$ in indeed noetherian, it may still be generated in various degrees. As a corollary of a result of Huneke's \cite[Corollary 2.5]{Huneke1986}, we know $P^{(n)} = P^n$ for all $n \geqslant 1$, or equivalently $\sra(P)$ is generated in degree $1$, exactly when $P$ is a complete intersection. In the language of section \ref{s:gtsvd}, we say that $\sra(P)$ has generation type $1$. Herzog and Ulrich characterized when the symbolic Rees algebra $\sra(P)$ is generated in degree up to $2$, or has generation type $2$, and showed that this implies that $P$ is self-linked \cite{MonCurvesGen2}. The cases when $\sra(P)$ has generation type $3$ \cite{NoethSymbReesAlgDegrees} and $4$ \cite{Degree4SpaceMonCurvesI,Degree4SpaceMonCurvesII} have also been completely characterized; these characterizations are all in terms of the Hilbert-Burch matrix of $P$.

With the subject of finite generation presenting such a difficult problem, the literature on other ring-theoretic properties of $\sra(I)$ is not as vast. Watanabe \cite{Watanabe} asked whether $\sra(I)$ must be Cohen-Macaulay whenever it is noetherian, where $I$ is a divisorial ideal in a strongly F-regular ring $R$. Watanabe constructed an example \cite[Example 4.4]{Watanabe} of a divisorial ideal $I$ in an F-rational ring whose Rees algebra is noetherian but not Cohen-Macaulay. When $R$ is strongly F-regular, Singh showed that the answer to Watanabe's question is affirmative provided that a certain auxiliary ring is finitely generated over $R$ \cite{Singh}. The construction of this auxiliary ring is an iterated symbolic Rees algebra.

In positive characteristic, the symbolic Rees algebra of the canonical module $\omega = \omega_R$ of a local, normal, complete ring $R$ plays an important role in studying Frobenius actions on the injective hull of the residue field. A significant construction in this context is the anticanonical cover $\bigoplus_{n \geqslant 0} {\rm Hom}_R(\omega^{(n)},R)$.  The number of generators for this ring as an algebra over $R$, if finite, bounds the Frobenius complexity of $R$ as shown by Enescu and Yao \cite{EnescuYao}.

The research and literature surrounding  symbolic Rees algebras is abundant and growing at a steady rate. While we cannot do complete justice to this topic  by presenting an exhaustive review, we expand in some directions which are closest to our interests in the following sections.

\section{Criteria for noetherianity}\label{subsection criteria}

In this section we discuss criteria for finite generation, and equivalently noetherianity, of symbolic Rees algebras and structural invariants of finitely generated symbolic Rees algebras. 

\subsection{Noetherianity}

The most comprehensive criterion, described below in \Cref{noetherian equivalences} $(4)\Leftrightarrow (1)$, states that, under mild hypotheses, finite generation of a symbolic Rees algebra $\R_s(I)$ is equivalent to the fact that there exists a Veronese subalgebra $\bigoplus_{n\geqslant 0} I^{(kn)}t^{kn}$ isomorphic to the (ordinary) Rees algebra 
$\R(I^{(k)})=\bigoplus_{n\geqslant 0} (I^{(k)})^nt^{n}$. An equivalent assertion is that a Veronese subalgebra of $\R_s(I)$ admits a standard grading.

The various parts of the following criterion appear in different places in the literature: the equivalence of (1) and (3) is developed in \cite{Rees1958ProblemZariski} and (4) appears in work of Schenzel \cite[Theorem 1.3]{SchenzelFiltrations}. We include a proof since this result is central to our discussion.

\begin{proposition}[Standard graded subalgebra criterion]
\label{noetherian equivalences}
	Let $R$ be a noetherian ring and $I$ an ideal in $R$. The following are equivalent:
	\begin{enumerate}[(1)]
		\item $\sra(I)$ is a finitely generated $R$-algebra.
		\item $\sra(I)$ is a noetherian ring.
		\item There exists $d$ such that for all $n \geqslant 1$,
		$$I^{(n)} \quad = \sum_{a_1 + 2 a_2 + \cdots + d a_d = n} I^{a_1} \left( I^{(2)} \right)^{a_2} \cdots \left( I^{(d)} \right)^{a_d}.$$
	\end{enumerate}
Furthermore, when these equivalent conditions hold, then
\begin{enumerate}[(4)]
		\item There exists $k$ such that $I^{(kn)} = \left( I^{(k)} \right)^n$ for all $n \geqslant 1$.
	\end{enumerate}
Conditions $(1) - (4)$ are equivalent whenever $R$ is an excellent ring.
\end{proposition}

\begin{proof}
	The fact that $(1)$ implies $(2)$ is a consequence of Hilbert's Basis Theorem. Moreover, since $\sra(I)$ is an $\mathbb{N}$-graded algebra and $\sra(I)_0 = R$ is a noetherian ring, the equivalence between $(1)$ and $(2)$ is a general fact about graded $R$-algebras; see for example \cite[Proposition 1.5.4]{BrunsHerzog} for a proof.
	Statement $(3)$ says that $\sra(I)$ is generated in degree up to $d$ as an $R$-algebra, and thus is equivalent to $(1)$.
	
	To show that $(3)$ implies $(4)$, we follow \cite[Lemma 2]{Rees1958ProblemZariski}, where in fact a stronger statement is proved. We will show that $k$ can in fact be taken to be $k = d \cdot d!$.
	
	First, suppose that $n \geqslant k$. For each choice of $a_1 + 2 a_2 + \cdots + d a_d = n \geqslant d \cdot d!$, we must have $i a_i \geqslant d!$ for some $i$, by the pigeonhole principle. Moreover, $q \colonequals \frac{d!}{i}$ is an integer, so
	$$I^{a_1} \left( I^{(2)} \right)^{a_2} \cdots \left( I^{(d)} \right)^{a_d} = \left( I^{(i)} \right)^{q} I^{a_1} \left( I^{(2)} \right)^{a_2} \cdots \left( I^{(i)} \right)^{a_i - q} \cdots \left( I^{(d)} \right)^{a_d} \subseteq I^{(d!)} I^{(n-d!)}.$$
	In particular, $I^{(n)} \subseteq I^{(d!)} I^{(n-d!)}$ for all $n \geqslant d \cdot d!$, but since $ I^{(d!)} I^{(n-d!)} \subseteq I^{(n)}$ holds because symbolic powers form a graded family, in fact we have shown that $I^{(n)} = I^{(d!)} I^{(n-d!)}$.
	
	Now consider any $n \geqslant 1$. Since $n k \geqslant k = d \cdot d!$, then
	$$I^{(kn)} = I^{(d!)} I^{(kn-d!)} = \left( I^{(d!)} \right)^2 I^{(kn-2d!)} = \cdots = \left( I^{(d!)} \right)^{d} I^{(kn-d \cdot d!)} \subseteq I^{(d \cdot d!)} I^{(kn-d \cdot d!)},$$
	so that 
	$$I^{(kn)} = I^{(d \cdot d!)} I^{(kn-d \cdot d!)} = I^{(k)} I^{(k(n-1))}.$$
	By induction, the statement follows.
	
	On the other hand, if $(4)$ holds, then the algebra
	$$A \colonequals \bigoplus_{n \geqslant 0} I^{(kn)} t^{kn} = \bigoplus_{n \geqslant 0} \left( I^{(k)} t^{k} \right)^n  \subseteq \sra(I) \subseteq R[t]$$
	is finitely generated. The fact that $(4)$ implies the remaining equivalent statements will follow once we show that $\sra(I)$ is a finitely generated algebra over $A$. To do that, we follow the argument in \cite[(2.2)]{SchenzelFiltrations}.
	
	Let $B$ denote the integral closure of $A$ inside $R[t]$. Recall\footnote{The book \cite{HunekeSwansonIntegral2006} is a comprehensive reference on the subject of integral closure.} that $B$ is the subring of $R[t]$ given as follows:
	$$B = \left\lbrace f \in R[t]: f^d + a_{d-1} f^{d-1} + \cdots + a_1 f + a_0 = 0 \textrm{ for some } f_i \in A \right\rbrace.$$
	We claim that $\sra(I) = \bigoplus I^{(n)} t^n \subseteq B$. To show that, consider $u \in I^{(i)} t^i$. Then
	$$u^k \in \left( I^{(i)} \right)^k t^{ik} \subseteq I^{(ki)} t^{ki} = \left( I^{(k)} t^k \right)^i,$$
	so that $u$ is a root of $T^k - u^k$. Since $u^k\in A$,  $u$ is integral over $A$, which implies that $u \in B$. Since $\sra(I)$ is generated by such elements, we conclude that $\sra(I) \subseteq B$. Moreover, $B$ is a finitely generated module over $A$, by \cite[Remark 12.3.11 or Theorem 9.2.2]{HunekeSwansonIntegral2006}. Therefore, $\sra(I)$ must be finitely generated over $A$ by the Artin-Tate theorem \cite{ArtinTate}.
\end{proof}

For \Cref{noetherian equivalences} $(4) \Rightarrow (3)$, the condition we need is that the integral closure of a finitely generated $R$-algebra $B$ in a finite extension is a finitely generated algebra over $B$; rings with this property are called Nagata rings (see \cite[Chapter 13]{Matsumura}). This holds whenever $R$ is excellent or analytically unramified, and in particular every polynomial or power series ring over a field has this property.

\begin{remark}
\label{value of k}
	The proof of  \Cref{noetherian equivalences} shows that when the symbolic Rees algebra is noetherian and generated in degree up to $d$, then for $k = d \cdot d!$, we do have $I^{(kn)} = \left( I^{(k)} \right)^n$ for all $n \geqslant 1$. In fact, it is shown in \cite[Lemma 2]{Rees1958ProblemZariski} that if the symbolic Rees algebra is generated in degrees $a_1, \ldots, a_s$, and $r$ is the least common multiple of $a_1, \ldots, a_s$, then we can take $k = sr$.
\end{remark}

Under mild assumptions, part $(3)$ of \Cref{noetherian equivalences} above might be rewritten, as follows:

\begin{lemma}\label{lemma k equivalence}
	Let $R$ be an excellent ring, and $I$ an ideal in $R$. Suppose that $k$ is such that $I^{(kn)} = \left( I^{(k)} \right)^n$ for all $n \geqslant 1$. Then there exists $A \geqslant 1$ such that for all $n \geqslant 1$, if $n = qk+r$, with $0 \leqslant r < k$, then
	$$I^{(n)} = \sum_{a=0}^A \left( I^{(k)} \right)^{q-a} I^{(ak+r)}.$$
\end{lemma}

\begin{proof}
	As before, note that the $R$-algebra 
	$$B \colonequals \bigoplus_{n \geqslant 0} I^{(kn)} t^{kn} = \bigoplus_{n \geqslant 0} \left( I^{(k)} t^{k} \right)^n  \subseteq R[t]$$
	is finitely generated, and that $\sra(I)$ is finitely generated over $B$.
	
	Suppose that $\sra(I)$ is generated over $B$ in degrees $a_1, \ldots, a_d$. Then
	$$\bigoplus_{n \geqslant 0} I^{(n)} t^n = \sra(I) = I^{(a_1)} t^{a_1} B \oplus \cdots \oplus I^{(a_d)} t^{a_d} B = \bigoplus_{m \geqslant 1} I^{(a_i)} \left( I^{(k)} \right)^m t^{a_i + km}.$$
	Finally, the theorem follows once we collect the pieces in degree $n$.
\end{proof}

While very useful, the criteria in \Cref{noetherian equivalences} often prove challenging to apply because they require checking infinitely many equalities of ideals. The next results of Huneke \cite[Theorems 3.1 and 3.25]{HunekeHilbertSymb} present ideal-theoretical criteria for the symbolic Rees ring $\R_s(P)$ of a height two prime $P$ of a 3-dimensional regular ring $R$ to be noetherian, which are relatively simple to apply. These criteria suffice to establish that every affine space curve of degree three as well as every monomial space curve of degree four have noetherian symbolic Rees algebras.

\begin{proposition}[Multiplicity criterion]
Let $R$ be a regular local ring with $\dim(R)=3$ and infinite residue field and let $P$ be a height two prime ideal of $R$. The following are equivalent:
\begin{enumerate}
\item $\R_s(P)$ is a finitely generated $R$-algebra.
\item There exist $k,l \geqslant 1$, $f\in P^{(k)}$, $g\in P^{(l)}$ and $x\not\in P$ such that 
$$\lambda(R/(f,g,x))=kl\lambda\left(R/(P+(x)\right).$$
\item There exist $f.g\in P$ such that $\sqrt{(f,g)}=P$ and the leading forms $f^*, g^*$ of $f,g$ in the associated graded ring  of $PR_P$ form a regular sequence.
\end{enumerate}
\end{proposition}

It is possible to extend this criterion to reduced ideals of height two that are not necessarily prime. For example, by \cite[Proposition 3.5]{HaHu}, if an ideal $I$ defines a set of $s$ points in $\P^2$ and if there exist $m\in \N$ and $f, g\in I^{(m)}$ such that $f,g$ form a regular sequence and $\deg(f)\deg(g)=m^2s$, then $\R_s(I)$ is a noetherian ring. This criterion can be applied to show that any set of $s\leqslant 8$ points in $\P^2$ gives rise to a noetherian symbolic Rees algebra. However, the converse implication is no longer valid, as shown in \cite{NagelSeceleanu}.

It turns out that the analytic spread $\ell(I)$ of $I$, which is defined to be the Krull dimension of the special fiber ring of $I$, $\R(I)/\m\R(I)$, plays an important role in the study of symbolic Rees algebras. Its contribution is due to work of McAdam on asymptotic primes of $I$. For the rest of this section, we assume $R$ is an excellent domain, although a weaker condition, locally quasi-unmixed, would suffice. Brodmann shows in \cite{BroAssympAssociatedPrimes} that the following set, known as the set of {\em asymptotic primes} of $I$, is finite:
\[
A^*(I) \colonequals \bigcup_{n\geqslant 0} \Ass(I^n).
\] 
McAdam \cite[Theorem 3]{McAdam} (see also \cite[Proposition 4.1]{McAdamBook}) shows that $P\in A^*(I)$ if and only if $\ell(IR_P)=\dim(R_P)$. 
Setting 
$$J=\bigcap_{P\in A^*(I)\setminus\Min(I)} P$$ 
yields another description of the symbolic powers of $I$ as saturations: $I^{(n)}=I^n:J^\infty$. We are now ready to state another criterion for finite generation of the symbolic Rees algebra. This appears for primary ideals in work of Katz and Ratliff \cite[Theorem A]{KatzRatliff}, and in the form presented here in \cite[Theorem 2.6]{CutkoskyHerzogSrinivasan}. 

\begin{proposition}[Analytic spread criterion]
Let $(R, \m)$ be an excellent domain. Then $\R_s(I)$ is a finitely generated $R$-algebra if and only if for all $P \in V (J)$ we have that   
\[\ell\left((I:J^\infty)R_P\right)< \dim(R_P).\] 
\end{proposition}

In a similar vein, Goto, Herrmann, Nishida, and Villamayor give a sufficient criterion for the symbolic Rees algebra to be Noetherian in terms of an equimultiplicity condition of some symbolic power. Their result in \cite[Theorem 3.3]{GHNV} states that if $\ell (I^{(n)})=\het(I^{(n)})$ for some natural number $n$ and ideal $I$ in an unmixed local ring, then $\R_s(I)$ is noetherian.

A large class of ideals with finitely generated symbolic Rees algebra is the class of monomial ideals. While none of the above criteria apply to show this, finite generation of the respective symbolic Rees algebras follows from  Gordan's lemma, which says that the set of all lattice points in a rational cone is a finitely generated affine semigroup. This approach is taken by Herzog, Hibi, and Trung in \cite[Proposition 1.4]{vertexcover}. For squarefree monomial ideals, finite generation of the symbolic Rees algebra was previously shown in work of Lyubeznik; see \cite[Proposition 1]{Lyubeznik}.

\subsection{Generation type and standard Veronese degree}
\label{s:gtsvd}
In this section, we explore the maximum degree of elements required to generate the symbolic Rees ring as an $R$-algebra and the minimum degree of a standard graded Veronese subalgebra. 

Following  Bahiano \cite{Bahiano}, we define
\begin{definition}
\label{def:generation type}
The {\em generation type} of a symbolic Rees algebra $\R_s(I)$ is the value
\[
\gt(R_s(I)) \colonequals\inf \{d \mid \R_s(I)=R[ It, I^{(2)}t^2 ,\ldots ,I^{(d)}t^d] \}.
\]
\end{definition}

Note that $\gt(R_s(I))\in \N\cup\{\infty\}$  and $\gt(R_s(I))\in \N$ if and only if $\R_s(I)$ is a noetherian ring. A challenging problem is to determine or bound this invariant for interesting classes of ideals.

\begin{problem}
Find effective bounds on $\gt(R_s(I))$, when finite, in terms of invariants of $I$.
\end{problem}
This problem has been studied predominantly in combinatorial contexts, when $I$ is a monomial ideal \cite{Bahiano, clutters, vertexcover}; it has also been studied for the ideal defining a space monomial curve \cite{NoethSymbReesAlgDegrees,Degree4SpaceMonCurvesI,Degree4SpaceMonCurvesII}, and in some cases of more general monomial curves, for example in \cite{DCruzDeg2}. For a monomial ideal $I$, finding a minimal set of algebra generators for $\R_s(I)$ translates into finding a Hilbert basis for an appropriate convex polyhedron \cite[Corollary 3.2]{clutters}. This is a computationally intensive problem, which can nevertheless be approached with the aid of specialized software \cite{Normaliz,4ti2}. 

When $I$ is a monomial ideal and $I^{(n)}=\overline{I^n}$ for each $n\in\N$, i.e.  when all the symbolic powers are the integral closures of the corresponding ordinary powers, then \cite[Corollary 3.11]{blowupmonomial} yields $\gt(R_s(I)) \leqslant \dim(R)-1$. When $I$ is the edge ideal of a simple graph, \cite{Bahiano} yields $\gt(R_s(I)) \leqslant (\dim(R)-1)(\dim(R)-\het(I))$. However, for arbitrary monomial ideals the best known bound seems to be given by \cite[Theorem 5.6]{vertexcover}:
\[
\gt(\R_s(I))\leqslant \frac{(\dim(R) + 1)^{(\dim(R)+3)/2}}{2^{\dim(R)}}.
\]

\Cref{noetherian equivalences} reveals the importance of standard graded Veronese subalgebras of the symbolic Rees algebra. We introduce a new invariant that captures the least degree where they occur.
\begin{definition}
\label{def:std Veronese degree}
The {\em standard Veronese degree} of an ideal $I$ is the value
\[
\svd(I) \colonequals\inf \{k \mid  (I^{(k)})^n=I^{(kn)} \text{ for all } n\in \N\}.
\]
\end{definition}
As before, $\svd(I)<\infty$ is equivalent to $\R_s(I)$ being a noetherian ring by \Cref{noetherian equivalences}, and the proof of this proposition yields the upper bound $\svd(I)\leqslant \gt(\R_s(I))\cdot \gt(\R_s(I))!$. \Cref{value of k} yields a sharper upper bound. For particular families of ideals, specific upper bounds can be found in the literature, for example for some space monomial curve families \cite{Morales} and for ideals defining Fermat-type point configurations \cite{NagelSeceleanu} (cf. \Cref{Fermat}).

We explore these invariants for a specific family of monomial ideals below, with an eye towards evaluating the optimality of these bound.

\begin{example}
\label{prop:symmshifted}
Let $n$ and $h\leqslant n-1$ be positive integers and let $I_{n,h}$ denote the following monomial ideal in the polynomial ring $R_n=k[x_1,\ldots,x_n]$ with coefficients in a field $k$
\[
I_{n,h} \colonequals \bigcap_{1\leqslant i_1<i_2<\cdots<i_h\leqslant n} (x_{i_1},x_{i_2}, \cdots, x_{i_h}).
\]
This family of ideals is known as {\em monomial star configurations}.
Then the following hold:
\begin{equation}
\label{eq:star}
\R_s(I)=R_n[x_{i_1}x_{i_2} \cdots x_{i_{n-h+m}} t^m, 1\leqslant m \leqslant h, i_1<i_2<\cdots<i_{n-c+m}\leqslant n],
\end{equation}
\[
\gt(\R_s(I_{n,h})) = h, \quad \text{ and } \quad 
\svd(I_{n,h}) \text{ is divisible by }  \lcm(1,2, \ldots,h).
\]
\end{example}

\begin{proof}
The description of the symbolic Rees algebra in \eqref{eq:star} is established using different notation  in \cite[Proposition 4.6]{vertexcover}. 

It follows that the generation type of this algebra is $h$, provided that the unique algebra generator of degree $h$, $\prod_{i=1}^n x_i\in I^{(h)}$, listed in \eqref{eq:star} cannot be decomposed as a product of squarefree monomials $m_1,m_2, \ldots, m_s$ with $s>1$, $m_i=x_{i,1}x_{i,2} \cdots x_{i,n-h+a_i} \in I^{(a_i)}$. This would yield $a_1+\cdots+a_s=h$  and because the degrees of these monomials are $\deg(m_i)=n-h+a_i$ we obtain the following impossible inequality
\[
\deg(m_1)+\cdots+\deg(m_s) \geqslant a_1+\cdots+a_s+s(n-h)=h+s(n-h)>n=\deg \prod_{i=1}^n x_i.
\]

Continuing to a discussion of the standard Veronese degree, let us first observe that the lowest degree of a nonzero element of $I_{n,h}^{(m)}$ is $\alpha(I_{n,h}^{(m)})=m+(n-h)\lceil\frac{m}{h}\rceil$. A simple calculation now verifies that when $\frac{m}{h}$ is not an integer, then $\alpha(I_{n,h}^{(mk)})<k\alpha(I_{n,h}^{(m)})=\alpha((I_{n,h}^{(m)})^k)$ whenever $k>h$. This restricts the possible values for $r\colonequals\svd(I_{n,h})$ to multiples of the height $h$. However, further restrictions on $r$ are imposed by consideration of the manner in which our family of ideals contracts  with respect to the inclusions $R_{n-i}\subset R_n$. Specifically, for all $h,n,m$ there are identities
\begin{eqnarray*}
I_{n,h}^{(m)}\cap R_{n-i} &=&\bigcap_{1\leqslant i_1<i_2<\cdots<i_h\leqslant n} (x_{i_1},x_{i_2}, \cdots, x_{i_h})^m \cap  R_{n-i}\\
&=&\bigcap_{j=h-i}^{h} \ \bigcap_{1\leqslant i_1<i_2<\cdots<i_j\leqslant n-i} (x_{i_1},x_{i_2}, \cdots, x_{i_j})^m \\
&=&  I_{n-i,h-i}^{(m)} \cap I_{n-i,h-i+1}^{(m)} \cap\cdots \cap  I_{n-i,h}^{(m)} \\
 &=& I_{n-i,h-i}^{(m)},
\end{eqnarray*}
where we make the convention that $I_{n,u}=R_n$ whenever $u<0$.
Similarly one deduces
\[
I_{n,h}^{m}\cap R_{n-i} =(I_{n,h}\cap R_{n-i})^m=I_{n-i,h-i}^m.
\]

If $I^{(rm)}=(I^{(r)})^m$ for all $m\in \N$ then for $0\leqslant i\leqslant h-1$ we deduce the identities
\begin{eqnarray*}
I_{n,h}^{(rm)}\cap R_{n-i} &=& (I_{n,h}^{(r)})^m \cap R_{n-i}, \text{ i.e.,} \\
I_{n-i,h-i}^{(rm)}&=& (I_{n-i,h-i}^{(r)})^m .
\end{eqnarray*}
By the previous reasoning, we see that $r$ must be divisible by all integers $1\leqslant h-i\leqslant h$, thus
$\lcm(1,2,\ldots, h)$ divides $r$. 
\end{proof}

We conjecture that for the family of ideals in Proposition \ref{prop:symmshifted} there is in fact an equality $\svd(I_{n,h}) =\lcm(1,2, \ldots,h)$. This prompts the following question:
\begin{question}
Can the bound in  \Cref{value of k} be improved for all monomial ideals $I$ to 
\[\svd(I)\leqslant \text{ the }\lcm \text{ of the degrees of any set of algebra generators for }\R_s(I)? \]
\end{question}

At this time we are unaware of any ideals that satisfy $\svd(I)<\gt(I)$. Hence we ask:
\begin{question}
Does the inequality $\gt(I) \leqslant \svd(I)$ hold for every ideal $I$?
\end{question}

\section{Applications to containment problems and asymptotic invariants}

\subsection{The Containment Problem}

Containments of the form $I^n \subseteq I^{(n)}$ are a direct consequence of \Cref{def:symbolicpower}, which further implies that $I^b \subseteq I^{(a)}$ if and only if $b\geqslant a$. Containments of the converse type $I^{(a)} \subseteq I^b$ are a lot more interesting. Together these form the basis for comparison of the ordinary and symbolic ideal topologies, which has been pioneered by Schenzel \cite{SchenzelContainmentProblem} and later Swanson \cite{Swanson}. This line of inquiry is nowadays known as the containment problem:

\begin{question}[Containment problem]
\label{containmentproblem}
 Let $R$ be a ring and let $I$ be an ideal of $R$ without embedded primes. For which pairs $a,b$ does the containment $I^{(a)} \subseteq I^b$ hold?
\end{question}

If for each value of $b$ there is a pair $a, b$ answering the above question, then the families $\lbrace I^{(n)} \rbrace_n$ and $\lbrace I^{n} \rbrace_n$ are cofinal, and induce equivalent topologies. In \cite{Swanson}, Swanson shows that the equivalence of ordinary and symbolic ideal topologies is linear, that is, if $\lbrace I^{(n)} \rbrace_n$ and $\lbrace I^{n} \rbrace_n$ are cofinal then there is a constant $c$, possibly depending on $I$, such that $I^{(cn)} \subseteq I^n$ for all $n \geqslant 1$. When the ambient ring is regular, this constant can be expressed explicitly, in terms of the \emph{big height} of $I$, the largest height of on associated prime of $I$. In fact, in this case the constant $c$ can even be taken uniformly, depending only on $R$, as shown by the following important results \cite{ELS,comparison,MaSchwede}.

\begin{theorem}[Ein--Lazarsfeld--Smith, Hochster--Huneke, Ma--Schwede]
\label{thm:containment}
 Let $R$ be a regular ring and $I$ an ideal in $R$. If $h$ is the big height of $I$, then $I^{(hn)} \subseteq I^n$ for all $n\geqslant 1$. 
  In particular, if $d=\dim(R)$, then $I^{((d-1)n)} \subseteq I^n$ for $n\geqslant 1$.
\end{theorem}

If we remove the regular assumption, and ask that $R$ be a complete normal local domain, it is still an open problem in general to determine whether there exists a uniform constant $c$, depending only on $R$, such that $P^{(cn)} \subseteq P^n$ for all $n \geqslant 1$ and all primes $P$. When $P$ is a prime ideal in a complete normal local domain, the $P$-symbolic and $P$-adic topologies are equivalent \cite{SchenzelPrimes}. More generally, if $R$ is an excellent Noetherian domain, the $P$-symbolic and $P$-adic topologies are equivalent for every prime $P$ if and only if going down holds between $R$ and its integral closure \cite{HKVcor}. Some of the recent progress on this problem in \cite{HKV,HKVFinExt,HKabelian,RobertHH,RobertUSTP,RobertNormal} is also described in some detail in \cite{SurveySP}.

More surprisingly, the containment problem is not settled even in the regular case. In fact, the containments provided by \Cref{thm:containment} are not necessarily best possible. In fact, examining the proof of the above theorem in \cite{comparison} one sees that in the case of positive characteristic, $\Char(R)=p$, it relies on containments of the form $I^{(hq)}\subseteq I^{[q]}$, where $q=p^e$ for $e\in \N$ and $I^{[q]}$ denotes the $q$th Frobenius power of $I$. The stronger containment $I^{(hq-q+1)}\subseteq I^{[q]}$ follows in this context using localization and the pigeonhole principle as explained in \cite[p.351]{comparison}. This yields the following improved containments:

\begin{proposition}
 Let $R$ be a regular ring, $I$ an ideal of $R$, and $h$ the big height of $I$. If $\Char(R)=p>0$, the containments $I^{(hq-h+1)}\subseteq I^q$ hold for $q=p^e$ and for each integer $e\geqslant 1$. 
\end{proposition}

This leads to the question of whether similar improvements can be carried over to arbitrary characteristic and arbitrary exponents. Harbourne proposed this as a conjecture in \cite{Seshadri, HaHu} for homogeneous ideals, which we write here for radical ideals.

\begin{conjecture}[Harbourne] 
\label{conj:Harbourne}
Let $I$ be a radical homogeneous ideal in a polynomial ring, and let $h$ be the big height of $I$. Then the containments $I^{(hn-h+1)}\subseteq I^n$ hold for all $n\geqslant 1$.
\end{conjecture}

\begin{remark}
To compare \Cref{conj:Harbourne} to \Cref{thm:containment}, it is instructive to note that \Cref{thm:containment} implies that $I^{(n)}\subseteq I^{\lfloor \frac{n}{h}\rfloor}$ for $n\geqslant 1$, while Harbourne's \Cref{conj:Harbourne} asks if $I^{(n)} \subseteq I^{\lceil \frac{n}{h}\rceil}$ for all $n\geqslant 1$.
\end{remark}

There are various cases where \Cref{conj:Harbourne} is known to hold: if $I$ is a monomial ideal \cite[Example 8.4.5]{Seshadri} or more generally if $I$ defines an F-pure ring \cite{GrifoHuneke}, if $I$ corresponds to a
general set of points in $\P^2$ \cite{BoH} or $\P^3$ \cite{Dumnicki2015}, and if $I$ defines a matroid configuration \cite{matroid}, that is, a union of codimension $c$ intersections of hypersurfaces such that any subset of at most $c+1$ of the equations of these hypersurfaces forms a regular sequence. Moreover, versions of \Cref{conj:Harbourne} hold in some singular settings as well \cite{GrifoMaSchwede}. Despite that, asking that \Cref{conj:Harbourne} holds for any radical ideal in a regular ring turns out to be too general.

\begin{example}[Dumnicki--Szemberg--Tutaj-Gasi\'nska \cite{counterexamples},  Harbourne--Seceleanu \cite{HaSeFermat}]
\label{Fermat}
 Let $n\geqslant 3$ be an integer and let $k$  be a field of with $\Char(k)\neq 2$  that contains $n$ distinct roots of unity. Let $R = k[x, y, z]$, and consider the ideal
\[I =(x(y^n-z^n),y(z^n-x^n),z(x^n-y^n))\]
defining $n^2+3$ points in $\P_k^2$, namely the 3 coordinates points together with the $n^2$ points defined by the complete intersection $(x^n-y^n, y^n-z^n)$. For this ideal $h=2$ but $I^{(3)}\not\subseteq I^2$, thus \Cref{conj:Harbourne} is not satisfied for $n=2$.
\end{example}

\Cref{Fermat} is particularly interesting because it shows that \Cref{conj:Harbourne} can fail even for ideals with noetherian symbolic Rees algebra. Indeed, in \cite{NagelSeceleanu} it is shown that the symbolic Rees algebras of the ideals in \Cref{Fermat} are finitely generated. On the other hand, the family of space monomial curves in \Cref{ex:monomialcurve}, which have (big) height 2, satisfy $I^{(4)}\subseteq I^3$ by \cite[Example 4.7]{GrifoStable}. This is a stronger containment than the one proposed by \Cref{conj:Harbourne}, and yet these ideals have non-noetherian symbolic Rees algebras.

  In \cite{HaSeFermat}, Harbourne and Seceleanu show that the containment $I^{(hn-h+1)} \subseteq I^n$ can fail for arbitrarily high values of $n$ that grow with the dimension of $R$, if $R$ is a polynomial ring of characteristic $p > 0$. However, in characteristic 0 all the known counterexamples to  \Cref{conj:Harbourne} found to date \cite{counterexamples, RealsCounterexample, ResurgenceKleinWiman, MalaraSzpond, NewCounterexample, DrabkinSeceleanu} are for the value $n=2$.
There are moreover no prime counterexamples to Harbourne's  \Cref{conj:Harbourne}. We emphasize this by asking:

\begin{question}[Harbourne conjecture for primes]
If $P$ is a prime ideal in a regular ring and $\het(P)=h$, then do the containments $P^{(hn-h+1)}\subseteq P^n$ hold for all $n\geqslant 1$?
\end{question}

For example, in characteristic other than $3$, it is know that all space monomial curves $(t^a, t^b, t^c)$ satisfy this containment for $n=2$ \cite[Theorem 4.1]{GrifoStable}, and also for $n \gg 0$ \cite[Corollary 4.3]{GHM}. There are also no known counterexamples to the following asymptotic version of Harbourne's conjecture formulated in \cite{GrifoStable}.

\begin{conjecture}[Stable Harbourne conjecture]
\label{conj:stableHarbourne}
Let $R$ be a regular ring and $I$ a radical ideal of $R$ with big height $h$. Then there exists $N>0$ such that the  containment $I^{(hn-h+1)}\subseteq I^n$ holds for all $n\geqslant N$.
\end{conjecture}

This stable version of Harbourne's Conjecture does hold for various classes of ideals in equicharacteristic rings, including examples with non-noetherian symbolic Rees algebra. In Subsection \ref{subsection asymptotic}, we will discuss ideals with expected resurgence, and all of these satisfy the Stable Harbourne Conjecture.

\subsection{Noetherian symbolic Rees algebras and the containment problem}
We now consider the implications of having a noetherian symbolic Rees algebra on the Containment Problem \ref{containmentproblem}. The first easy implication is that $I$ satisfies a version of Harbourne's Conjecture with the big height replaced by the generation type.

\begin{lemma}\label{harbourne d}
	Let $R$ be a noetherian ring and $I$ an ideal in $R$ with $\gt(I) = d$. Then for all $n \geqslant 1$,
	$$I^{(dn-d+1)} \subseteq I^n.$$
	In particular, $I$ satisfies Harbourne's Conjecture whenever $\gt(I) \leqslant \bight(I)$.
\end{lemma}

\begin{proof}
	Fix $n \geqslant 1$. By Lemma \ref{noetherian equivalences}, it is enough to show that for all choices of $a_1, \ldots, a_n \geqslant 0$ such that $a_1 + 2a_2 + 3a_3 + \cdots + d a_d = dn-d+1$,
	$$I^{a_1} \left( I^{(2)} \right)^{a_2} \cdots \left( I^{(d)} \right)^{a_d} \subseteq I^{dn-d+1}.$$
	To see this holds, note that $\left( I^{(i)} \right)^{a_i} \subseteq I^{a_i}$ for each $i$, so that
	$$I^{a_1} \left( I^{(2)} \right)^{a_2} \cdots \left( I^{(d)} \right)^{a_d} \subseteq I^{a_1 + a_2 + \cdots + a_d}.$$
	For each such choice of $a_1, \ldots, a_d$,
	$$d \left( a_1 + \cdots + a_d \right) \geqslant a_1 + 2a_2 + \cdots + d a_d = dn-d+1,$$
	so that
	$$a_1 + \cdots + a_d \geqslant \frac{d(n-1)+1}{d}.$$
	Since $a_1 + \cdots + a_d$ is an integer, we conclude that 
	$$a_1 + \cdots + a_d \geqslant (n-1)+1 = n.\qedhere$$
\end{proof}

Moreover, if $\sra(I)$ is noetherian, it suffices to check the containments for $n \leqslant \gt(I)$ in \Cref{conj:Harbourne} to conclude Harbourne's Conjecture holds for $I$:

\begin{lemma}\label{Harbourne up to d implies Harbourne always}
	Let $R$ be a noetherian ring and $I$ an ideal in $R$ such that $\gt(I) = d$. If $h$ is an integer such that
	$$I^{(i)} \subseteq I^{\left\lceil \frac{i}{h} \right\rceil}$$
	for all $i \leqslant d$, then for all $n \geqslant 1$,
	$$I^{(hn-h+1)} \subseteq I^n.$$
\end{lemma}

\begin{proof}
	Given $n \geqslant 1$,
	$$I^{(hn-h+1)} \quad = \sum_{a_1 + 2a_2 + \cdots + d a_d = hn-h+1} I^{a_1} \left( I^{(2)} \right)^{a_2} \cdots \left( I^{(d)} \right)^{a_d}.$$
	It is enough to show that for all $a_1, \ldots, a_d \geqslant 0$ such that $a_1 + 2 a_2 + \cdots + d a_d = hn-h+1$, the ideal
	$$J \colonequals I^{a_1} \left( I^{(2)} \right)^{a_2} \cdots \left( I^{(d)} \right)^{a_d}$$
	is contained in $I^n$. By assumption, $I^{(i)} \subseteq I^{\lceil \frac{i}{h} \rceil}$ for each $i$. Therefore, $J \subseteq I^N$, where
	$$N \geqslant \sum_{i=1}^d a_i \left\lceil \frac{i}{h} \right\rceil \geqslant \sum_{i=1}^d \frac{i a_i}{h} = \frac{hn-h+1}{h}.$$
	Since $N$ is an integer, we must have
	$$N \geqslant \left\lceil \frac{hn-h+1}{h} \right\rceil = n.\qedhere$$
\end{proof}

For example, as a consequence of \Cref{Harbourne up to d implies Harbourne always} and \cite[Theorem 4.4]{GrifoStable}, Harbourne's Conjecture holds for space monomial curves of generation type up to $6$.

In a similar vein, one may ask if the Stable Harbourne Conjecture holds when $\sra(I)$ is noetherian. Here is some evidence in that direction (cf. \cite[Theorem 5.28]{mythesis}).

\begin{theorem}
	Let $I$ be a radical ideal of big height $h$ in a regular ring $R$ containing a field. If $\svd(I)$ divides $h$, then $I^{(hn-h+1)} \subseteq I^n$ for all $n \gg 0$.
\end{theorem}

\begin{proof}
	First, notice there is nothing to show in the case when $h=1$, so we assume $h \geqslant 2$.
	By Lemma \ref{lemma k equivalence}, there exists an integer $A \geqslant 1$ such that for all $n \geqslant 1$,
	$$I^{(hn-h+1)} = \sum_{a=0}^A \left( I^{(h)} \right)^{n-1-a} I^{(ha+1)}.$$
	In prime characteristic $p$, consider $e$ such that $q \colonequals p^e \geqslant A + 1$. Whenever $n \geqslant q$, we have $hn-h+1 \geqslant hq-h+1 \geqslant hA+1$, and thus
	$$I^{(hn-h+1)} = \sum_{a=0}^A \left( I^{(h)} \right)^{n-1-a} I^{(ha+1)} = \sum_{a=0}^A \left( I^{(h)} \right)^{n-q} \left( I^{(h)} \right)^{q-1-a} I^{(ha+1)} \subseteq I^{(h(n-q))} I^{(hq-h+1)}.$$
	Now as we have mentioned above, $I^{(hq-h+1)} \subseteq I^{q}$ and $I^{(h(n-q))} \subseteq I^{n-q}$ by \cite{comparison}; the latter is true more generally, but the first statement requires specifically that we are in characteristic $p$ and $q = p^e$. Combining these two containments with the line above, we conclude that
	$$I^{(hn-h+1)} \subseteq I^{(hq-h+1)} I^{(h(n-q))} \subseteq I^{q} I^{n-q} = I^n.$$

	To prove the statement in equicharacteristic $0$, we need \cite[Theorem 1.2]{Hubl_monomial}, which says that there exists $N > 0$ such that $\left( I^{(2)} \right)^n \subseteq I^{n+1}$ for all $n \geqslant N$. Fix such $N$, and let $n \geqslant N + A + 1$. Then
	$$I^{(hn-h+1)} = \sum_{a=0}^A \left( I^{(h)} \right)^{n-1-a} I^{(ha+1)} = \sum_{a=0}^A \left( I^{(h)} \right)^{N} \left( I^{(h)} \right)^{n-1-N-a} I^{(ha+1)}.$$
	By \cite{ELS,comparison}, $I^{(ha+1)} \subseteq I^{a}$. Moreover, $\left( I^{(h)} \right)^{n-1-N-a} \subseteq I^{n-1-N-a}$ since $I^{(h)} \subseteq I$. By choice of $N$, $\left( I^{(h)} \right)^N \subseteq \left( I^{(2)} \right)^N \subseteq I^{N+1}$. Therefore,
	$$I^{(hn-h+1)} = \sum_{a=0}^A \left( I^{(h)} \right)^{N} \left( I^{(h)} \right)^{n-1-N-a} I^{(ha+1)} \subseteq \sum_{a=0}^A I^{N+1} I^{n-1-N-a} I^{a} = I^n.\qedhere$$
\end{proof}

So if the big height of an ideal $I$ is divisible by its standard Veronese degree, then $I$ satisfies the stable Harbourne \Cref{conj:stableHarbourne}. The ideals in \Cref{prop:symmshifted}, for example, do not have this property although they satisfy \Cref{conj:stableHarbourne}. Thus we ask:

\begin{question}
Which ideals satisfy the condition that $\bight(I)$ is divisible by $\svd(I)$?
\end{question}

In prime characteristic, there are other cases where the noetherianity of $\sra(I)$ implies the Stable Harbourne Conjecture \ref{conj:stableHarbourne}; for example, see \cite[Theorem 5.19 and Theorem 5.23]{mythesis}.

\subsection{Asymptotic invariants}\label{subsection asymptotic}

One way to study symbolic powers and the containment problem is through the development of asymptotic invariants. This is an idea pioneered by Bocci and Harbourne in \cite{BoH, resurgence2} with the definition of the resurgence of an ideal, and extended in \cite{asymptoticResurgence} with the definition of the asymptotic resurgence. We present these invariants and their relationship to the symbolic Rees algebra below.

\begin{definition}
 The {\em resurgence} of  an ideal $I$ and the {\em asymptotic resurgence} are given, respectively,  by 
 \[
 \rho(I)=\sup\left \{\frac{a}{b} \mid I^{(a)}\not\subseteq I^b \right \} \text{ and } \widehat{\rho}(I)=\sup\left \{\frac{a}{b} \mid I^{(at)}\not\subseteq I^{bt} \text{ for }t\gg0 \right \}. 
 \]
 \end{definition}

The importance of (asymptotic) resurgence to containment problems lies in the fact that, by definition, if $a,b$ are positive integers with $a > \rho(I)b$, then $I^{(a)} \subseteq I^b$. 

If $I$ is an ideal of a regular ring and has big height $h$, \Cref{thm:containment} implies that $1\leqslant \rho (I) \leqslant h$ and since the definitions yield $\widehat{\rho}(I)\leqslant \rho(I)$ we deduce that $1\leqslant \widehat{\rho} (I) \leqslant h$ as well.
If we have equality of ordinary and symbolic powers $I^{(n)} = I^n$ for all $n \geqslant 1$, then $\widehat{\rho}(I)=\rho(I) = 1$. However, the resurgence attaining its lowest possible value of 1 does not guarantee equality of  the ordinary and symbolic powers. 

\begin{example}[DiPasquale--Drabkin \cite{DiPasqualeDrabkin}]
The ideal $I=(abc, aef, cde, bdf)$ of the polynomial ring $R=k[a,b,c,d,e,f]$ satisfies $\rho(I)=1$ and $I^{(n)} = I^n +(abcdef)I^{n-2}$ for $n\geqslant 2$. In particular, the ordinary and symbolic powers do not coincide for any $n \geqslant 2$.
\end{example}

At the other end of the spectrum, whether the (asymptotic) resurgence attains its largest possible value equal to the big height has implications on the stable Harbourne \Cref{conj:stableHarbourne}. First, it follows easily from the definition that \Cref{conj:stableHarbourne} holds for ideals with $\rho(I) < \bight(I)$ (see \cite[Remark 2.7]{GrifoStable}); moreover, it is sufficient to show that $\widehat\rho(I) < \bight(I)$.

\begin{theorem}[Grifo--Huneke--Mukundan {\cite[Proposition 2.11]{GHM}}]
 Let $I$ be a radical ideal in either a regular local ring containing a field, or a quasi-homogeneous radical ideal in a polynomial ring over a field. If $\widehat{\rho}(I)< \bight(I)$, then the containment $I^{(hn-h+1)}\subseteq I^n$ holds for all $n\gg 0$.
\end{theorem}

Ideals satisfying $\rho(I) < \bight(I)$ have been termed {\em ideals with expected resurgence} in \cite{GHM}. Classes of ideals with expected resurgence include: those defining general points in $\mathbb{P}^n$ \cite[Theorem 4.2]{ChudnovskyGeneralPoints}, locally complete intersection ideals $I$ a polynomial ring that are minimally generated by forms of degree lower than $\bight(I)$ \cite[Theorem 3.1]{GHM}, ideals $I$ of a local or standard graded regular ring $(R,\m,k)$ which contains a field so that $R/I$ is Gorenstein, $I^{(n)}=I^n:\m^\infty$ and either $k$ has positive characteristic or the symbolic Rees algebra of I is noetherian \cite{GHM2}.

Because of these considerations it becomes important to develop methods for determining the (asymptotic) resurgence of an ideal. In order to do this, it is helpful to investigate another asymptotic invariant.

\begin{definition}
\label{def:waldschmidt}
Let $I$ be an ideal of a graded ring and denote by $\alpha(I)$ the smallest degree of a nonzero homogeneous from in $I$. The {\em Waldschmidt constant} of $I$ is the value 
\[
\widehat{\alpha}(I)=\lim_{n\to \infty}\frac{\alpha(I^{(n)})}{n}=\inf_{n}\frac{\alpha(I^{(n)})}{n}.
\]
\end{definition}

For homogeneous ideals of a polynomial ring, the following inequalities discovered by Bocci and Harbourne often hold the key to computing resurgence.

\begin{theorem}[Bocci--Harbourne {\cite[Theorem 1.2.1]{BoH}}]
Let $I$ be a homogeneous ideal of a polynomial ring. Then there is an inequality 
$$\frac{\alpha(I)}{\widehat{\alpha}(I)}\leqslant \rho(I).$$
If, in addition, $I$ defines a 0-dimensional subscheme, then 
$$\rho (I) \leqslant \reg(I)/\widehat{\alpha}(I),$$
where $\reg(I)$ denotes the Castelnuovo-Mumford regularity of $I$.
\end{theorem}

However, resurgences and Waldschmidt constants remain elusive invariants. In the case of ideals having Noetherian symbolic Rees algebras, however, one can get a better handle on these invariants by expressing them in terms of finitely many symbolic power of ideals.

\begin{theorem}[{Drabkin--Guerrieri \cite[Theorem 3.6]{DrabkinGuerrieri}, DiPasquale--Drabkin \cite[Proposition 2.2, Corollary 3.6]{DiPasqualeDrabkin}}]
Suppose $I$ is an ideal of a polynomial ring which has noetherian symbolic Rees algebra. Then the Waldschmidt constant,  asymptotic resurgence, and the resurgence of $I$ are rational numbers and can be computed as follows
\begin{eqnarray*}
\widehat{\alpha}(I) &=&\min_{n\leqslant \gt(I)}\frac{\alpha(I^{(n)})}{n},  \\
\widehat{\rho}(I) &=&\max_{1\leqslant i\leqslant r, 1\leqslant j \leqslant \gt(I)}\left\{\frac{j\nu_i(I)}{\nu_i(I^{(j)})}\right \}, \text{ and } \\
\rho(I) &=&\begin{cases}
\max_{(a,b)\in \text{finite set}}\left\{ \frac{a}{b} \mid I^{(a)}\not\subseteq I^b\right\} & \text{ if } \rho(I)\neq \widehat{\rho}(I)\\ 
 \widehat{\rho}(I) & \text{otherwise},
\end{cases}
\end{eqnarray*}
where $\nu_1,\ldots, \nu_r$ denote the distinct Rees valuations\footnote{ For details on Rees valuations and their applications the reader is invited to consult \cite[\S10.1]{HunekeSwansonIntegral2006}.} of $I$ and the finite set in the last displayed equation is given explicitly  in  \cite[Proposition 2.2]{DiPasqualeDrabkin}.
\end{theorem}

\begin{remark}
If $I$ is an ideal of a polynomial ring which has noetherian symbolic Rees algebra, then in fact the Waldschmidt constant is determined by a single symbolic power corresponding to the standard Veronese degree. Indeed,  \Cref{def:waldschmidt} and \Cref{noetherian equivalences}(3) yield $\widehat{\alpha}(I)=\alpha(I^{(\svd(I))})/\svd(I)$.
\end{remark}
 
Ideals with irrational values of the Waldschmidt constant are expected to abound. Indeed, Nagata's conjecture \cite{Nagata59} would imply that the Waldschmidt constant of a radical ideal $I$ defining $s$ general points in $\P^2$ is $\widehat{\alpha}(I)=\sqrt{s}$, often producing an irrational value. However, no examples of ideals with confirmed irrational Waldschmidt constant, resurgence, or asymptotic resurgence have been constructed yet. Thus we propose the following task. 

\begin{problem}
Provide examples of ideals with irrational Waldschmidt constant, resurgence, or asymptotic resurgence.
\end{problem}

The lower bound $\widehat{\alpha}(I) \geqslant \alpha(I)/(d-1)$ holds for homogeneous ideals in a $d$ dimensional polynomial ring, and it follows easily from the containments in \Cref{thm:containment}. The details can be found in \cite{HaHu}, but the lower bound itself --- although phrased in a different language --- appears in work of Waldschmidt \cite{Waldschmidt} and Skoda \cite{Skoda}. Improvements on this lower bound have been proposed by Chudnovsky \cite{Chudnovsky} and Demailly \cite{Demailly} in relation to the difficult question of finding the least degree of a homogeneous polynomial vanishing at a given set of points in projective space to a prescribed order. The validity of the bounds suggested by Chudnovsky and Demailly follows if one can establish containments of the symbolic power ideals deeper within the ordinary powers than provided by \Cref{thm:containment}. We make these containments precise in \Cref{q:contChudDemailly} below, while also abstracting the bounds suggested by Chudnovsky and Demailly to the more general setting of homogeneous radical ideals in \Cref{q:ChudDemailly}.

\begin{question}
\label{q:contChudDemailly}
Let $I$ be either a radical ideal of big height $h$ in a regular local ring $(R, \m)$, or a homogeneous radical ideal of big height $h$ in a polynomial ring $R$ with maximal homogeneous ideal $\mathfrak{m}$. Do the following containments
\begin{equation}
\label{contChud}
I^{(rh)} \subseteq \m^{r(h-1)}I^r
\end{equation}
\begin{equation}
\label{contDemailly}
I^{(r(h+m-1))} \subseteq \m^{r(h-1)} \left({I^{(m)}}\right)^r
\end{equation}
hold for all $m,r\geqslant 1$?
\end{question}

Note that \eqref{contChud} is the particular case of \eqref{contDemailly} with $m=1$. These first appeared as a question for ideals of points in \cite[Question 4.2.3]{HaHu}, and the more general version for radical ideals of big height $h$ appeared in \cite[Conjecture 2.9]{CEHH2017}. Both containments are satisfied for squarefree monomial ideals by \cite[Corollary 4.3]{CEHH2017}, where in fact a stronger statement was shown \cite[Theorem 4.2]{CEHH2017}. Similar containment for the defining ideal of a general set of points in $\mathbb{P}^2$ were investigated in \cite{BCH2014}.

The validity of the containments in equations \eqref{contChud} and \eqref{contDemailly} of \Cref{q:contChudDemailly} would imply the bounds for Waldschmidt constants of  homogeneous radical ideals given in \eqref{ineqChud} and \eqref{ineqDemailly} respectively of the following question.

\begin{question}[Chudnovsky and Demailly type bounds on the Waldschmidt constant]
\label{q:ChudDemailly}
Let $I$ be a homogeneous radical ideal of big height $h$ in a polynomial ring $R$. Do the following inequalities
\begin{equation}
\label{ineqChud}
\frac{\alpha(I^{(n)})}{n} \geqslant \frac{\alpha(I) + h -1}{h} \text{ and thus }\widehat \alpha(I)\geqslant \frac{\alpha(I) + h -1}{h} 
\end{equation}
\begin{equation}
\label{ineqDemailly}
\frac{\alpha(I^{(n)})}{n}\geqslant \frac{\alpha(I^{(m)}) + h - 1}{m+h-1}   \text{ and thus }\widehat \alpha(I)\geqslant \frac{\alpha(I^{(m)}) + h -1}{h}
\end{equation}
hold for all $n, m\geqslant 1$?
\end{question}

An affirmative answer to \Cref{q:ChudDemailly} \eqref{contChud} has been given for ideals defining general points in $\P^2$ in \cite{HaHu}, for ideals defining general sets of projective points of sufficiently large cardinality in  \cite{MTGChudnovsky}, for very general sets of points in arbitrary projective spaces in \cite{FMXChudnovsky}, and for ideals defining sufficiently many general sets of points in projective space in \cite{ChudnovskyGeneralPoints}, where \Cref{q:contChudDemailly}  \eqref{contChud} is shown to hold for $r \gg 0$. \Cref{q:ChudDemailly} \eqref{ineqDemailly} is answered in the affirmative for general points in $\P^2$ by Esnault and Viehweg \cite{EsnaultViehweg} and for very general sets of projective points of sufficiently large cardinality in arbitrary projective spaces by work of Malara, Szemberg and Szpond \cite{MSSDemailly}, extended by Chang and Jow \cite{CJDemailly}. Recently, an affirmative answer to \Cref{q:ChudDemailly} \eqref{ineqDemailly} has also been established for sufficiently large general sets of points in arbitrary projective spaces by Bisui, Grifo, H\`a and Nguy$\tilde{\text{\^e}}$n in \cite{DemaillyGeneralPoints}, where an affirmative answer to \Cref{q:contChudDemailly} \eqref{contDemailly} is also provided in the same context for infinitely many values of $r$, although not for all $r$ or even $r \gg 0$. Outside of the context of points, the answer to all of these questions is also affirmative for generic determinantal ideals and the defining ideals of star configurations in any codimension \cite{DemaillyGeneralPoints}.

However, both of the above questions remain open in the form stated here, and are open even for  ideals having noetherian symbolic Rees algebra. 

\begin{remark}
Suppose that $I$ has a finitely generated symbolic Rees algebra. Since the limit in the definition of $\widehat\alpha(I)$ exists, we obtain
$$\widehat\alpha(I) = \lim_{n \rightarrow \infty} \frac{\alpha(I^{(\svd(I) n)})}{\svd(I)n} = \frac{\alpha(I^{(\svd(I))})}{\svd(I)}.$$
Alternatively, since $\widehat\alpha(I)$ is also given as an infimum, one can compute $\widehat\alpha(I)$ by taking

$$\widehat\alpha(I) = \min \left\lbrace \alpha(I), \frac{\alpha(I^{(2)})}{2}, \ldots, \frac{\alpha(I^{\gt(I)})}{\gt(I)} \right\rbrace.$$
As a consequence, the containments \eqref{contChud} and \eqref{contDemailly} of \Cref{q:ChudDemailly} can be reduced to checking only those instances with $r\leqslant \gt(I)$ or, alternatively, only the case $r=\svd(I)$. Similarly, the inequalities \eqref{ineqChud} and \eqref{ineqDemailly} of \Cref{q:ChudDemailly} reduce to checking 
\[
\frac{\alpha(I^{(n)})}{n} \geqslant \frac{\alpha(I) + h -1}{h} \text{ and }\frac{\alpha(I^{(n)})}{n} \geqslant \frac{\alpha(I^{(m)}) + h -1}{m+h-1} \text{ for } 1\leqslant n\leqslant \gt(I) \text{ and } m\geqslant 1
\]
or, equivalently, 
\[
 \frac{\alpha(I^{(\svd(I))})}{\svd(I)} \geqslant \frac{\alpha(I) + h -1}{h} \text{ and } \frac{\alpha(I^{(\svd(I))})}{\svd(I)} \geqslant \frac{\alpha(I^{(m)}) + h -1}{m+h-1} \text{ for }  m\geqslant 1.
\]
\end{remark}

\vspace{2em}

\paragraph{ \bf Acknowledgements}
The first author is supported by NSF grant DMS-2001445. The second author is supported by NSF grant DMS-1601024.
We thank Craig Huneke for help with tracking down the history of the terminology ``symbolic Rees algebra", and Jos\'e Gonzalez for discussions regarding Mori dream spaces. We also thank Thomas Polstra for pointing us to \cite{HKVcor}, and Elena Guardo for finding a typo in a previous version of the paper. Finally, we thank Anurag Singh for very detailed comments on an earlier version.


\newcommand{\etalchar}[1]{$^{#1}$}
\providecommand{\bysame}{\leavevmode\hbox to3em{\hrulefill}\thinspace}
\providecommand{\MR}{\relax\ifhmode\unskip\space\fi MR }
\providecommand{\MRhref}[2]{%
  \href{http://www.ams.org/mathscinet-getitem?mr=#1}{#2}
}
\providecommand{\href}[2]{#2}

\end{document}